\documentclass[a4paper,11pt,draft]{article}
\usepackage{amsmath,amssymb,enumerate,color}
\usepackage{amsthm,color}
\usepackage{txfonts}

\newcommand{\R}{\mathbb{R}}
\newcommand{\C}{\mathbb{C}}

\newcommand{\N}{\mathbb{N}}
\newcommand{\ep}{\varepsilon}
\newcommand{\pa}{\partial}

\DeclareMathOperator{\lifespan}{LifeSpan}

\newtheorem{theorem}{Theorem}[section]
\newtheorem{lemma}[theorem]{Lemma}

\theoremstyle{remark}
\newtheorem{remark}{Remark}[section]
\theoremstyle{definition}

\newtheorem{definition}{Definition}[section]

\numberwithin{equation}{section}

\makeatletter
\def\@cite#1#2{[{{\bfseries #1}\if@tempswa , #2\fi}]}
\makeatother                  
\begin{document}
\begin{center}
\Large{{\bf
Remark on 
upper bound for lifespan 
of solutions to 
semilinear \\
evolution equations
in a two-dimensional exterior domain}}
\end{center}

\vspace{5pt}

\begin{center}
Masahiro Ikeda%
\footnote{
Department of Mathematics, Faculty of Science and Technology, Keio University, 3-14-1 Hiyoshi, Kohoku-ku, Yokohama, 223-8522, Japan/Center for Advanced Intelligence Project, RIKEN, Japan, 
E-mail:\ {\tt masahiro.ikeda@keio.jp/masahiro.ikeda@riken.jp}}
and
Motohiro Sobajima%
\footnote{
Department of Mathematics, 
Faculty of Science and Technology, Tokyo University of Science,  
2641 Yamazaki, Noda-shi, Chiba, 278-8510, Japan,  
E-mail:\ {\tt msobajima1984@gmail.com}}
\end{center}

\newenvironment{summary}{\vspace{.5\baselineskip}\begin{list}{}{%
     \setlength{\baselineskip}{0.85\baselineskip}
     \setlength{\topsep}{0pt}
     \setlength{\leftmargin}{12mm}
     \setlength{\rightmargin}{12mm}
     \setlength{\listparindent}{0mm}
     \setlength{\itemindent}{\listparindent}
     \setlength{\parsep}{0pt}
     \item\relax}}{\end{list}\vspace{.5\baselineskip}}
\begin{summary}
{\footnotesize {\bf Abstract.}
In this paper we consider 
the following initial-boundary  value problem 
with the power type nonlinearity 
$|u|^p$ with $1<p\leq 2$
in a two-dimensional  exterior domain
\begin{equation}\label{our}
\begin{cases}
\tau \pa_t^2 u(x,t)-\Delta u(x,t)+e^{i\zeta}\pa_t u(x,t)=\lambda |u(x,t)|^p, 
& (x,t)\in \Omega \times (0,T),
\\
u(x,t)=0, 
& (x,t)\in \pa\Omega \times (0,T),
\\
u(x,0)=\ep f(x), 
& x\in \Omega,
\\
\pa_t u(x,0)=\ep g(x), 
& x\in \Omega,
\end{cases}
\end{equation}
where $\Omega$ is given by $\Omega=\{x\in\R^2\;;\;|x|>1\}$, 
$\zeta\in [-\frac{\pi}{2}, \frac{\pi}{2}]$, 
$\lambda\in\C$ and $\tau\in \{0,1\}$ switches 
the parabolicity, dispersivity and hyperbolicity. 
Remark that $2=1+2/N$ is well-known as 
the Fujita exponent. If $p>2$, then 
there exists a small global-in-time solution
of \eqref{our} for some initial data small enough 
(see Ikehata \cite{Ikehata05exterior}), 
and if $p<2$, then global-in-time solutions 
cannot exist for any positive initial data
(see Ogawa--Takeda \cite{OT09} and Lai--Yin \cite{LY17}).
The result is that for given initial data $(f,\tau g)\in H^1_0(\Omega)\times L^2(\Omega)$ satisfying 
$(f+\tau g)\log |x|\in L^1(\Omega)$ with some requirement, 
the solution blows up at finite time, 
and moreover, the upper bound for lifespan of solutions 
to \eqref{our} is given as the following {\it double exponential type} when $p=2$:
\[
\lifespan(u)\leq \exp[\exp(C\ep^{-1})] . 
\]
The crucial idea is to use test functions which approximates the harmonic function 
$\log |x|$ satisfying Dirichlet boundary condition and 
the technique for derivation of lifespan estimate in 
\cite{IkedaSobajima3}. 
}
\end{summary}



\section{Introduction}
In this paper we consider 
the following initial-boundary  value problem 
with the power type nonlinearity 
$|u|^p$ with $1<p\leq 2$
in a two-dimensional  exterior domain
\begin{equation}\label{ndw}
\begin{cases}
\tau \pa_t^2 u(x,t)-\Delta u(x,t)+e^{i\zeta}\pa_t u(x,t)=\lambda |u(x,t)|^p, 
& (x,t)\in \Omega \times (0,T),
\\
u(x,t)=0, 
& (x,t)\in \pa\Omega \times (0,T),
\\
u(x,0)=\ep f(x), 
& x\in \Omega,
\\
\pa_t u(x,0)=\ep g(x), 
& x\in \Omega,
\end{cases}
\end{equation}
where $\Omega$ is given by $\Omega=\{x\in\R^2\;;\;|x|>1\}$, 
$\zeta\in [-\frac{\pi}{2}, \frac{\pi}{2}]$, 
$\lambda\in\C$ and $\tau\in \{0,1\}$ switches 
the parabolicity, dispersivity and hyperbolicity. 

First we give some comment for the case of whole space $\R^N$. 
For the semilinear heat equation $\pa_tu-\Delta u =u^{p}$, 
Fujita \cite{Fujita66} found blowing up solutions 
with small initial data when $p<p_F(N)=1+\frac{2}{N}$. 
The exponent $p_F(N)$ is well-known as the Fujita exponent. 
There are many subsequent papers about further detailed analysis
(see e.g. Hayakawa \cite{Hayakawa73}, 
Kobayashi--Shirao--Tanaka \cite{KST77}, 
Mizoguchi--Yanagida \cite{MY96,MY98}, 
Fujishima--Ishige \cite{FI11,FI12} and their references therein). 
Then Lee--Ni \cite{LN92} 
gave a precise estimate for lifespan of solutions to 
$\pa_tu-\Delta u =u^{p}$ with $u(0)=\ep f\geq 0$ as 
\[
\lifespan(u)\sim 
\begin{cases}
C\ep^{-(\frac{1}{p-1}-\frac{N}{2})^{-1}}
&\text{if}\ 1<p<p_F(N), 
\\
\exp(C\ep^{-(p-1)})
&\text{if}\ p=p_F(N)
\end{cases}
\]
by using the structure of the heat kernel in the whole space.

For the Schr\"odinger equation without gauge invariance 
$i\pa_tu-\Delta u =\lambda|u|^{p}$, 
the study of blowup phenomena for $L^1$-initial data 
has been done in the literature 
(see Ikeda--Wakasugi \cite{IW13}, Fujiwara--Ozawa \cite{FO16} 
and Oh--Okamoto--Pocovnicu \cite{OOParxiv}). 

For the semilinear damped wave equation without gauge invariance 
$\pa_t^2u-\Delta u+\pa_tu =\lambda|u|^{p}$, 
Li--Zhou \cite{LZ95} proved blowup phenomena with lifespan estimates
for the solution with arbitrary small initial data 
when $N=1,2$ and $1<p\leq p_F(N)$. 
Then the same question for $N=3$ is solved by 
Nishihara \cite{Nishihara03Ibaraki}. 
For general dimensional case, 
Todorova--Yordanov \cite{TY01} established existence of 
global solutions with sufficiently small initial data
if $p>p_F(N)$. In the critical case $p=p_F(N)$, 
Zhang \cite{Zhang01} proved blowup phenomena 
for the solution with arbitrary small initial data. 
Ikeda--Ogawa \cite{IO16} gave an estimates of lifespan, 
but the gap between lower and upper bounds was not filled. 
Then Lai--Zhou \cite{LZarxiv} gives the precise 
lifespan estimate for blowup solution with small initial data for $p=p_F(N)$. 

Later, our previous paper \cite{IkedaSobajima3}
deals with all of the above prototype of the evolution equations 
simultaneously and 
proves the upper estimates 
of lifespan of solutions to \eqref{ndw} when $1<p\leq p_F(N)$ 
as the same upper bound as Lee--Ni \cite{LN92};
note that in \cite{IkedaSobajima3}, 
a kind of space-dependent damping in a cone-like domain 
is also dealt with. 

The problem for exterior domain also has been done
for many mathematicians.
Tsutsumi \cite{Tsutsumi83} proved the global existence of 
solutions to the Schr\"odinger equation with gauge invariance 
for $N\geq 3$ and $p\in 2\N$.
Kobayashi--Misawa \cite{KM13} 
deals with nonlinear heat equation 
$\pa_t u -\Delta u = (1+|x|\log (B|x|))^{-1}|u|^{p-1}u$ 
with $N=2$ via Hardy and BMO spaces 
and posed that the critical exponent for 
the nonlinearity is $p=\frac{3}{2}$. 

If we focus our attention to 
the result for $N=2$ with the nonlinearity $|u|^p$, 
Ikehata \cite{Ikehata05exterior} constructed 
small global-in-time solutions 
of nonlinear damped wave equation when $p>2=p_F(2)$. 
If $p<2$, then global-in-time solutions 
cannot exist for compactly supported initial data
(see Ogawa--Takeda \cite{OT09} and Lai--Yin \cite{LY17}).

However, in the authors' knowledge 
there is no previous works dealing with the blowup phenomena 
with lifespan estimates
for \eqref{ndw} in exterior domains 
via test function method which is applicable 
to the critical case and also to the Schr\"odinger equation. 
The aim of this paper is the following. 
For given initial data $(f,\tau g)\in H^1_0(\Omega)\times L^2(\Omega)$ satisfying $(f+\tau g)\log |x|\in L^1(\Omega)$ with some requirement, 
we shall prove that the solution blows up at finite time, 
and moreover, the upper bound for lifespan of solutions 
to \eqref{ndw} with $p=2$ given as the following {\it double exponential type}:
\[
\lifespan(u)\leq \exp[\exp(C\ep^{-1})] . 
\]
The crucial idea is to use test functions which approximates the harmonic function 
$\log |x|$ satisfying Dirichlet boundary condition and 
the technique for derivation of lifespan estimate in 
\cite{IkedaSobajima3}. 

To state our main result, 
we give a definition of solutions to \eqref{ndw} 
as follows.
\begin{definition}\label{def:sol}
We say that 
$u$ is a solution of \eqref{ndw} with $\tau=0$ in $[0,T)$ if 
\[
u\in C([0,T);H_0^1(\Omega))
\cap 
L^p_{\rm loc}(\overline{\Omega}\times [0,T))
\]
with $u(x,0)=\ep f(x)$ and
for every $\psi\in C^2(\Omega\times [0,T)$ with 
${\rm supp}\,\psi\subset\subset \overline{\Omega}\times [0,T))$ 
and $\psi|_{\pa\Omega}=0$ 
\begin{align*}
&
\ep 
e^{i\zeta}\int_{\Omega}
  f(x)\psi(x,0)
\,dx
+
\lambda
\int_0^T
\int_{\Omega}
  |u(x,t)|^p\psi(x,t)
\,dx
\,dt
\\
&
=\int_0^T
\int_{\Omega}
  \Big(\nabla u(x,t)\cdot\nabla \psi(x,t)
   -e^{i\zeta} u(x,t)\pa_t\psi(x,t)
  \Big)
\,dx
\,dt.
\end{align*}
Similarly, 
$u$ is a solution of \eqref{ndw} with $\tau=1$ in $[0,T)$ if 
\[
u\in C([0,T);H_0^1(\Omega))
\cap 
C^1([0,T);L^2(\Omega))
\cap 
L^p_{\rm loc}(\overline{\Omega}\times [0,T))
\]
with $u(x,0)=\ep f(x)$ and
for every $\psi\in C^2(\Omega\times [0,T))$ with 
${\rm supp}\,\psi\subset\subset \overline{\Omega}\times [0,T)$ 
and $\psi|_{\pa\Omega}=0$ 
\begin{align*}
&
\ep 
\int_{\Omega}
  g(x)\psi(x,0)  
\,dx
+
\int_0^T
\int_{\Omega}
  |u(x,t)|^p\psi(x,t)
\,dx
\,dt
\\
&=
\int_0^T
\int_{\Omega}
  \Big(
  \nabla u(x,t)\cdot\nabla \psi(x,t)
  -\pa_t u(x,t)\pa_t\psi(x,t)
  +a(x)\pa_t u(x,t)\psi(x,t)\Big)
\,dx
\,dt.
\end{align*}
\end{definition}
We remark that 
the existence and uniqueness of local-in-time solutions to 
\eqref{ndw} in this sense 
can be verified by the standard way 
via Dirichlet Laplacian $-\Delta$ endowed with the domain 
$H^2(\Omega)\cap H^1_0(\Omega)$ 
and the Gagliardo--Nirenberg inequality 
(see e.g., Cazenave--Haraux \cite{CHbook} 
and Cazenave \cite{Cazenavebook}). 
Then we introduce the lifespan of the solution $u$. 
\begin{definition}\label{def:lifespan}
We denote $\lifespan(u)$ as the maximal existence time of solutions to respective problem \eqref{ndw}. Namely, 
\[
\lifespan(u)=
\sup \{T>0\;;\;\text{$u$ is a unique weak solution of \eqref{ndw} in $[0,T)$}\}. 
\]
\end{definition}

Now we are in a position to state the main result 
in this paper. 

\begin{theorem}\label{main}
Let $\tau\in\{0,1\}$, $\zeta\in [-\frac{\pi}{2},\frac{\pi}{2}]$, $\lambda\in\C\setminus\{0\}$, $1<p<\infty$ and let $u$ be a 
unique solution of \eqref{ndw}. If $p\leq 2$ and 
$f(x)\log |x|, \tau g(x)\log |x|\in L^1(\Omega)$ with 
\[
\int_{\mathcal{C}_{\Sigma}}
  (\tau g(x)+e^{i\zeta}f(x))
\log |x|\,dx\notin\{-\rho\lambda\in \C\;;\;\rho\geq 0\},
\]
then $\lifespan(u)<\infty$. Moreover, 
$\lifespan(u)$ has the following upper bound: 
\begin{align*}
\lifespan(u)
\leq 
\begin{cases}
C\ep^{-\frac{p-1}{2-p}}(\log \ep^{-1})^{p-1}
&\text{if}\ 1<p<2,
\\
\exp\exp\Big(C\ep^{-1}\Big)
&\text{if}\ p=2.
\end{cases}
\end{align*}
\end{theorem}

\begin{remark}
If $N\geq 3$, then we can take $1-|x|^{2-N}$ (instead of $\log |x|$) as a 
harmonic function satisfying Dirichlet boundary condition on $\pa\Omega$. 
Using this function as a test function, 
we can prove 
\begin{align*}
\lifespan(u)
\leq 
\begin{cases}
C\ep^{-(\frac{1}{p-1}-\frac{N}{2})^{-1}}
&\text{if}\ 1<p<1+\frac{2}{N},
\\
\exp\Big(C\ep^{-(p-1)}\Big)
&\text{if}\ p=1+\frac{2}{N}.
\end{cases}
\end{align*}
This is exactly the same as the case of $\Omega=\R^N$. 
The exterior problem in two-dimension 
seems quite exceptional. 
\end{remark}

The present paper is organized as follows. 
In Section 2, we will prepare 
our cutoff functions and give 
a key property for deriving the upper bound of 
the lifespan of solutions to \eqref{ndw} 
(see Lemma \ref{key} below); 
remark that the essence of the construction 
of family of cutoff functions are due to 
our previous paper \cite{IkedaSobajima3}. 
Section 3 is devoted to prove Theorem \ref{main}
by applying Lemma \ref{key}.

\section{Preliminaries}

The choice of cutoff functions 
is similar to \cite{IkedaSobajima3}.
We fix two kinds of functions $\eta\in C^2([0,\infty))$ 
and $\eta^*\in L^\infty((0,\infty))$ 
as follows, which will be used in the 
cut-off functions:
\[
\eta(s)
\begin{cases}
=1& \text{if}\ s\in [0, 1/2],
\\
\textrm{is decreasing}& \text{if}\ s\in (1/2,1),
\\
=0 & \text{if}\ s\in [1,\infty), 
\end{cases}
\quad
\eta^*(s)
=
\begin{cases}
0& \text{if}\ s\in [0, 1/2),
\\
\eta(s)& \text{if}\ s\in [1/2,\infty).
\end{cases}
\]
\begin{definition}\label{psi}
For $p>1$, we define for $R>0$, 
\begin{align*}
\psi_R(x,t)
&=
[\eta(s_R(x,t))]^{2p'}, 
\quad 
(x,t)\in \Omega\times [0,\infty),
\\
\psi_R^*(x,t)
&=
[\eta^*(s_R(x,t))]^{2p'}, 
\quad 
(x,t)\in \Omega\times [0,\infty).
\end{align*}
with 
\[
s_R(x,t)=\frac{(|x|-1)^2+t}{R}.
\]
We also set
$P(R)=\left\{(x,t)\in \Omega\times [0,\infty)\;;\;
(|x|-1)^2+t\leq R\right\}$. 
\end{definition}
\begin{remark}
The choice $s_R=R^{-1}[(|x|-1)^2+t]$ is essential 
for our problem in an exterior domain
in the sense of \eqref{eq:log-compatible} (see below). 
\end{remark}
Then we have the following lemma.
\begin{lemma}\label{lem:psi}
Let $\psi_R$ and $\psi_R^*$ be as in Definition \ref{psi}.
Then $\psi_R$ satisfies the following properties:
\begin{itemize}
\item[\bf (i)] If $(x,t)\in P(R/2)$, then $\psi_R(x,t)=1$, 
and if $(x,t)\notin P(R)$, then $\psi_R(x,t)=0$.
\item[\bf (ii)] 
There exists a positive constant $C_1$ such that 
for every $(x,t)\in P(R)$, 
\begin{align*}
|\pa_t \psi_R(x,t)|\leq C_1R^{-1}[\psi_R^*(x,t)]^{\frac{1}{p}}.
\end{align*}
\item[\bf (iii)] 
There exists a positive constant $C_2$ such that 
for every $(x,t)\in P(R)$, 
\begin{align*}
|\pa_t^2 \psi_R(x,t)|
\leq 
C_2R^{-2}[\psi_R^*(x,t)]^{\frac{1}{p}}.
\end{align*}
\item[\bf (iv)] 
There exists a positive constant $C_3$ such that 
for every $(x,t)\in P(R)$, 
\begin{align*}
|\nabla \psi_R(x,t)|
\leq C_3R^{-1}|x|(\log|x|)[\psi_R^*(x,t)]^{\frac{1}{p}}.
\end{align*}
\item[\bf (v)] 
There exists a positive constant $C_4$ such that 
for every $(x,t)\in P(R)$, 
\begin{align*}
|\Delta \psi_R(x,t)|\leq C_4R^{-1}
[\psi_R^*(x,t)]^{\frac{1}{p}}.
\end{align*}
\end{itemize}
\end{lemma}
\begin{proof}
All of the assertions follow from the direct calculation 
by noticing
\[
\pa_ts_R=
\frac{1}{R}, \quad
\nabla s_R=\frac{2}{R}\left(1-\frac{1}{|x|}\right)x, 
\quad
\Delta s_R=\frac{2}{R}\left(2-\frac{1}{|x|}\right)
\]
and 
\begin{align}\label{eq:log-compatible}
1
-\frac{1}{|x|}\leq \log |x|
\end{align}
for $x\in \Omega$. 
\end{proof}

The following lemma is the key assetion of 
the present paper which is 
similar as 
{\cite[Lemma 2.10]{IkedaSobajima3}}, 
but the situation with logarithmic function 
is included. 

\begin{lemma}\label{key}
Let $\delta>0$, $C_0>0$, $R_1>0$, $\theta\geq 0$, $\kappa\in \R$
and $0\leq w\in L^1_{\rm loc}([0,T);L^1(\Omega))$
for $T>R_1$. Assume that 
for every $R\in [R_1,T)$, 
\begin{align}\label{criterion}
\delta 
+
\iint_{P(R)}
  w(x,t)\psi_R(x,t)
\,dx\,dt
\leq 
C_0R^{-\frac{\theta}{p'}}(\log R)^{\frac{\kappa}{p'}}
\left(
\iint_{P(R)}
  w(x,t)\psi_R^*(x,t)
\,dx\,dt
\right)^{\frac{1}{p}}.
\end{align}
Then $T$ has to be bounded above as follows:
\begin{align*}
T\leq 
\begin{cases}
C\delta^{-\frac{1}{\theta}}(\log(\delta^{-1}))^{\frac{\kappa}{\theta}}
&\text{if}\ \theta>0, \ \kappa\in\R,
\\[5pt]
\exp\left(C\delta^{-\frac{p-1}{1-\kappa(p-1)}}\right)
&\text{if}\ \theta=0, \ \kappa<\frac{1}{p-1},
\\[5pt]
\exp\exp\left(C\delta^{-(p-1)}\right)
&\text{if}\ \theta=0, \ \kappa=\frac{1}{p-1}.
\end{cases}
\end{align*}
\end{lemma}

\begin{proof}[Proof of Lemma \ref{key}]
We set 
\begin{align*}
y(r):=
\iint_{P(r)}
  w(x,t)\psi_r^*(x,t)
\,dx\,dt, \quad r\in (0,T), 
\end{align*}
Then as in \cite[Lemma 2.10]{IkedaSobajima3}, 
we have 
\begin{align*}
\int_{0}^{R}
y(r)
r^{-1}\,dr
&=
\iint_{P(R)}
  w(x,t)
\left(
\int_{((|x|-1)^2+t)/R}^{\infty}
  \left[
  \eta^*\left(s\right)
  \right]^{2p'}
s^{-1}\,ds
\right)
\,dx
\,dt
\\
&\leq 
(\log 2)\iint_{P(R)}
  w(x,t)
  \psi_R(x,t)
\,dx
\,dt.
\end{align*}
Taking 
\[
Y(R)=
\int_{0}^R y(r)r^{-1}\,dr, \quad \rho\in (R_1, T), 
\]
we deduce from \eqref{criterion} that 
for $R\in  (R_1,T)$, 
\begin{align*}
\left(
\delta +\frac{1}{\log 2}
Y(R)
\right)^{p}
&\leq 
C_0^p
R^{-\theta (p-1)}
(\log R)^{\kappa (p-1)}
\iint_{P(R)}
  w(x,t)\psi_R^*(x,t)
\,dx\,dt
\\
&=
C_0^p
R^{1-\theta (p-1)}
(\log R)^{\kappa (p-1)}
Y'(R).
\end{align*}
Putting 
\[
Y(R)=Z\left(\int_{\log R_1}^{\log R} e^{\theta(p-1) s}s^{-\kappa (p-1)}\,ds\right), 
\quad 
0<\rho<\rho_T=\int_{\log R_1}^{\log T} e^{\theta(p-1) s}s^{-\kappa (p-1)}\,ds.
\]
implies that
\begin{equation}
\label{odi}
\frac{d}{d\rho}\Big((\log 2)\delta+Z(\rho)\Big)^{1-p}
\leq -(p-1)(\log 2)^{-p}C_1^{-p}, 
\quad 
\rho\in (0,\rho_T).
\end{equation}
Integrating it over $[\rho_1,\rho_2]\subset (0,\rho_T)$, we  obtain
\[
\rho_2
<
\rho_1+(p-1)^{-1}(\log 2)C_1^p\delta^{-(p-1)}.
\]
Letting $\rho_2 \uparrow \rho_T$ and $\rho_1 \downarrow 0$, 
we find 
\[
\rho_T=\int_{\log R_1}^{\log T} e^{\theta(p-1) s}s^{-\kappa (p-1)}\,ds
\leq 
(p-1)^{-1}(\log 2)C_1^p\delta^{-(p-1)}.
\]
If the function $e^{\theta(p-1) s}s^{-\kappa (p-1)}$ is not integrable at infinity, 
then $T$ has to be finite. 
More precisely, the asymptotics of $\rho_T$ 
for large $T$ is as follows, respectively. 
If $\theta>0$, then 
$\rho_T\approx \frac{1}{(p-1)\theta}
T^{\theta(p-1)}(\log T)^{-\kappa (p-1)}$. 
If $\theta=0$ and $\kappa<\frac{1}{p-1}$, then
$\rho_T\approx \frac{1}{1-\kappa (p-1)}
(\log T)^{1-\kappa (p-1)}$. 
In the rest case $\theta=0$ and $\kappa=\frac{1}{p-1}$, 
we have 
$\rho_T\approx \log \log T$. 
These imply the desired bounds of $T$ for sufficiently small 
$\delta>0$. 
\end{proof}

\section{Proof of Theorem \ref{main}}

We show the assertion for $\tau=0$ and $\tau=1$ simultaneously. 
By the definition of weak solution to \eqref{ndw}, 
we can verify that 
for every $\psi\in C^2(\overline{\Omega}\times[0,T))$ with 
${\rm supp}\,\psi\subset\subset \overline{\Omega}\times [0,T)$ 
and $\psi|_{\pa\Omega}=0$
\begin{align*}
&
\ep 
\int_{\Omega}
  \tau g\psi(0)  
  +
  f\Big(
  -\tau \pa_t\psi(0)
  +e^{i\zeta}\psi(0)\Big)
\,dx
+
\lambda 
\int_0^T
\int_{\Omega}
  |u(t)|^p\psi(t)
\,dx
\,dt 
\\
&=
\int_0^T
\int_{\Omega}
  u(t)\Big(
  -\Delta \psi(t)
  + \tau\pa_t^2\psi(t)
  -e^{i\zeta}\pa_t\psi(t)\Big)
\,dx
\,dt.
\end{align*}
Noting that 
\begin{align*}
\lim_{R\to\infty}
\int_{\Omega}
  \tau g\Phi \psi_R(0)  
  +
  f\Phi\Big(
  -\tau \pa_t\psi_R(0)
  +e^{i\zeta}\psi_R(0)\Big)
\,dx
=
\int_{\Omega}
  \Big(\tau g+e^{i\zeta}f\Big)
\Phi\,dx
\end{align*}
with $\Phi(x)=\log|x|$, 
we see that there exist $R_0>0$, $\xi_0\in (-\frac{\pi}{2},\frac{\pi}{2})$ and $c_0>0$ such that 
for every $R\geq R_0$,
\[
{\rm Re}
\left[
\lambda^{-1}e^{i\xi}\int_{\Omega}
  \tau g\Phi \psi_R(0)  
  +
  f\Phi\Big(
  -\tau \pa_t\psi_R(0)
  +e^{i\zeta}\psi_R(0)\Big)
\,dx
\right]\geq c_0.
\]
Now we assume that $\lifespan(u)>R_0$. 
Since $\Phi$ is independent of $t$, it follows from 
Lemma \ref{lem:psi} that 
\begin{align*}
&
\Big|-\Delta (\Phi\psi_R(t))
+\tau \pa_t^2(\Phi\psi_R(t))
-e^{i\zeta}\pa_t(\Phi\psi_R(t))\Big|
\\
&=
\Big|-2\nabla \Phi\cdot\nabla \psi_R(t)
-\Phi\Delta\psi_R(t)
+\Phi\pa_t^2\psi_R(t)
-e^{i\zeta}\Phi\pa_t\psi_R(t)\Big|
\\
&
\leq 
+
\frac{2C_3}{R}\Phi[\psi_R^*(t)]^{\frac{1}{p}} 
+
\frac{C_4}{R}\Phi[\psi_R^*(t)]^{\frac{1}{p}} 
+
\frac{\tau C_2}{R^2}\Phi[\psi_R^*(t)]^{\frac{1}{p}} 
+
\frac{C_1}{R}\Phi[\psi_R^*(t)]^{\frac{1}{p}} 
\\
&\leq 
\frac{C_5}{R}\Phi[\psi_R^*(t)]^{\frac{1}{p}}
\end{align*}
with $C_5=2C_3+C_4+\tau C_2 R_0^{-1}+C_1$. 
Therefore choosing the 
test function $\psi(x,t)=\Phi(x)\psi_R(x,t)$
which satisfies the Dirichlet boundary condition,  
we obtain 
\begin{align*}
&c_0\ep
+\cos \xi\iint_{P(R)}
   |u(t)|^p\Phi\psi_R(t)
\,dx
\,dt
\\
&\leq 
{\rm Re}
\left[
\lambda^{-1}e^{i\xi_0}
\iint_{P(R)}
   u(t)\Big(\tau \pa_t^2(\Phi\psi_R(t))-\Delta (\Phi\psi_R(t))-e^{i\zeta}\pa_t(\Phi\psi_R(t))\Big)
\,dx
\,dt
\right]
\\
&\leq 
\frac{C_5}{|\lambda| R}
\iint_{P(R)}
|u|\Phi[\psi_R^*(t)]^{\frac{1}{p}}
\,dx
\,dt
\\
&\leq 
\frac{C_5}{|\lambda|R}
\left(
\iint_{P(R)}
\Phi
\,dx
\,dt
\right)^{\frac{1}{p'}}
\left(
\iint_{P(R)}
|u(t)|^p
\Phi\psi_R^*(t)
\,dx
\,dt
\right)^{\frac{1}{p}}.
\end{align*}
Therefore noting that 
\[
\iint_{P(R)}
\Phi
\,dx
\,dt
\leq 
\pi R
(\sqrt{R}+1)^2\log (\sqrt{R}+1).
\]
we deduce
\begin{align*}
&c_0\ep
+\iint_{P(R)}
  |u(t)|^p\Phi\psi_R(t)
\,dx\,dt
\leq 
C_6
R^{1-\frac{2}{p}}(\log R)^{\frac{1}{p'}}
\left(
\iint_{P(R)}
  |u(t)|^p\Phi\psi_R^*(t)
\,dx\,dt
\right)^{\frac{1}{p}}
\end{align*}
for some positive constant $C_6>0$. 
Applying Lemma \ref{key} with $w=|u|^p\Phi$, we have 
the desired estimate for $\lifespan(u)$. 
The proof is complete. 
\qed

\subsection*{Acknowedgements}
This work is partially supported 
by Grant-in-Aid for Young Scientists Research (B) 
No.16K17619 
and 
by Grant-in-Aid for Young Scientists Research (B) 
No.15K17571.



\begin{thebibliography}{30}

\bibitem{Cazenavebook}
T. Cazenave, 
``Semilinear Schr\"odinger equations,'' 
Courant Lecture Notes in Mathematics {\bf 10}, 
New York University, Courant Institute of Mathematical Sciences, New York; 
American Mathematical Society, Providence, RI, 2003.


\bibitem{CHbook} 
    T. Cazenave, A. Haraux, 
    ``An introduction to semilinear evolution equations,'' 
    Translated from the 1990 French original by Yvan Martel and revised by the authors. 
    Oxford Lecture Series in Mathematics and its Applications {\bf 13}. 
    The Clarendon Press, Oxford University Press, New York, 1998.

\bibitem{FI11}
Y. Fujishima, K. Ishige, 
{\it Blow-up for a semilinear parabolic equation with large diffusion on $\R^N$},
J.\ Differential Equations {\bf 250} (2011), 2508--2543.

\bibitem{FI12}
Y. Fujishima, K. Ishige, 
{\it Blow-up for a semilinear parabolic equation with large diffusion on $\R^N$. II}, 
J.\ Differential Equations {\bf 252} (2012), 1835--1861.
 
\bibitem{Fujita66}
H. Fujita, 
{\it On the blowing up of solutions of the Cauchy problem for $u_t=\Delta u+u^{1+\alpha}$}, 
J.\ Fac.\ Sci.\ Univ.\ Tokyo Sect.\ I {\bf 13} (1966), 109--124.

\bibitem{FO16}
K. Fujiwara, T. Ozawa, 
{\it Finite time blowup of solutions to the nonlinear Schr\"odinger equation without gauge invariance}, 
J.\ Math.\ Phys.\ {\bf 57} (2016), 082103, 8 pp.

\bibitem{Hayakawa73}
K. Hayakawa, 
{\it On nonexistence of global solutions of some semilinear parabolic differential equations}, 
Proc.\ Japan Acad.\ {\bf 49} (1973), 503--505.

\bibitem{II15}
M. Ikeda, T. Inui, 
{\it Small data blow-up of $L^2$ or $H^1$-solution for the semilinear Schr\"odinger equation without gauge invariance}, 
J.\ Evol.\ Equ.\ {\bf 15} (2015), 571--581.

\bibitem{IO16}
    M. Ikeda, T. Ogawa, 
    {\it Lifespan of solutions to the damped wave equation with a critical nonlinearity},
    J.\ Differential Equations {\bf 261} (2016), 1880--1903.

\bibitem{IkedaSobajima3}
M. Ikeda, M. Sobajima, 
{\it Upper bound for lifespan 
of solutions to certain semilinear parabolic, 
dispersive and hyperbolic equations  
via a unified test function method}, \texttt{arXiv:1710.06780}.

\bibitem{IW13}
M. Ikeda, Y. Wakasugi, 
{\it Small-data blow-up of $L^2$-solution for the nonlinear Schr\"odinger equation without gauge invariance}, 
Differential Integral Equations {\bf 26} (2013), 1275--1285.

\bibitem{IW15}
M. Ikeda, Y. Wakasugi, 
{\it A note on the lifespan of solutions to the semilinear damped wave equation},
Proc.\ Amer.\ Math.\ Soc.\ {\bf 143} (2015), 163--171.

\bibitem{Ikehata05exterior}
R. Ikehata, 
{\it Two dimensional exterior mixed problem for semilinear damped wave equations},  J.\ Math.\ Anal.\ Appl.\ {\bf 301} (2005), 366--377. 

\bibitem{KM13}
T. Kobayashi, M. Misawa, 
{\it $L^2$ boundedness for the $2$D exterior problems for the semilinear heat and dissipative wave equations}, 
Harmonic analysis and nonlinear partial differential equations, 
1--11, RIMS K$\hat{\rm o}$ky$\hat{\rm u}$roku Bessatsu, {\bf B42}, 
Res.\ Inst.\ Math.\ Sci.\ (RIMS), Kyoto, 2013.

\bibitem{KST77}
K. Kobayashi, T. Sirao, H. Tanaka, 
{\it On the growing up problem for semilinear heat equations}, 
J.\ Math.\ Soc.\ Japan {\bf 29} (1977), 407--424.

\bibitem{LY17}
    N. Lai, S. Yin, 
    {\it Finite time blow-up for a kind of initial-boundary value problem of semilinear damped wave equation}, 
    Math.\ Methods Appl.\ Sci.\ {\bf 40} (2017), 1223--1230. 

\bibitem{LZarxiv}
N.-A. Lai, Y. Zhou, 
{\it The sharp lifespan estimate for semilinear damped wave equation with Fujita critical power in high dimensions}, \texttt{arXiv:1702.07073}.

\bibitem{LN92}
T.-Y. Lee, W.-M. Ni, 
{\it Global existence, large time behavior and life span of solutions of a semilinear parabolic Cauchy problem}, 
Trans.\ Amer.\ Math.\ Soc.\ {\bf 333} (1992), 365--378.

\bibitem{LZ95}
T.T. Li, Y. Zhou, 
{\it Breakdown of solutions to 
$\square u+u_t=|u|^{1+\alpha}$}, 
Discrete Contin.\ Dynam.\ Systems {\bf 1} (1995), 503--520.

\bibitem{MY96}
N. Mizoguchi, E. Yanagida, 
{\it Blow-up of solutions with sign changes for a semilinear diffusion equation}, 
J.\ Math.\ Anal.\ Appl.\ {\bf 204} (1996), 283--290.

\bibitem{MY98}
N. Mizoguchi, E. Yanagida, 
{\it Blowup and life span of solutions for a semilinear parabolic equation}, 
SIAM J.\ Math.\ Anal.\ {\bf 29} (1998), 1434--1446.

\bibitem{Nishihara03Ibaraki}
K. Nishihara, $L^p$-$L^q$ estimates for the 3-D damped wave equation and their application to the semilinear problem, Seminar Notes of Math. Sci., 6, Ibaraki Univ., 2003, 69--83.

\bibitem{OOParxiv}
T. Oh, M. Okamoto, O. Pocovnicu, 
{\it On the probabilistic well-posedness of the nonlinear Schr\"odinger equations with non-algebraic nonlinearities}, \texttt{arXiv:1708.01568}.

\bibitem{OT09}
T. Ogawa, H. Takeda, 
{\it Non-existence of weak solutions to nonlinear damped wave equations in exterior domains}, Nonlinear Anal.\ {\bf 70} (2009), 3696--3701.

\bibitem{TY01}
G. Todorova, B. Yordanov, 
{\it Critical exponent for a nonlinear wave equation with damping}, 
J.\ Differential Equations {\bf 174} (2001), 464--489.

\bibitem{Tsutsumi83}
Y. Tsutsumi, 
{\it Global solutions of the nonlinear Schr\"odinger equation in exterior domains}, 
Comm.\ Partial Differential Equations {\bf 8} (1983), 1337--1374. 

\bibitem{Zhang01}
Q.S. Zhang, 
{\it A blow-up result for a nonlinear wave equation with damping: the critical case}, 
C.\ R.\ Acad.\ Sci.\ Paris S\'er.\ I Math.\ {\bf 333} (2001), 109--114.

\end{thebibliography}

\end{document}